\documentclass[12pt]{amsart}
\usepackage[margin=1in]{geometry}
\usepackage[utf8]{inputenc}
\usepackage{graphicx}
\usepackage{amsthm}
\usepackage{hyperref}
\usepackage{url}
\usepackage{lipsum}
\usepackage{amsmath}
\usepackage{amsfonts}
\usepackage{verbatim}
\usepackage[english]{babel}
\usepackage[abbrev,msc-links]{amsrefs}
\numberwithin{equation}{section}
\newcommand{\la}{\left|}
\newcommand{\ra}{\right|}
\newcommand{\lp}{\left(}
\newcommand{\rp}{\right)}

\newcommand{\lc}{\left[}
\newcommand{\rc}{\right]}
\newcommand{\ds}{\displaystyle}
\newcommand{\ZZ}{\mathbb{Z}}

\newcommand{\RR}{\mathbb{R}}

\newcommand{\ts}{\textstyle}
\allowdisplaybreaks
 
\newtheorem{theorem}[equation]{Theorem}
\newtheorem{lemma}[equation]{Lemma}

\newtheorem{proposition}[equation]{Proposition}

\theoremstyle{definition}

\title{Some properties of generalized $k$-Pell sequences} 
\author{Elijah Soria} 
\address{Department of Mathematics and Computer Science, Saint Mary's College of California, 1928 Saint Mary's Rd., Moraga, CA 94575}
\email{eps4@stmarys-ca.edu} 
\date{\today}
\thanks{This research was funded by the 2015 School of Science Summer Research Program at Saint Mary's College of California.}
\begin{document}
\maketitle
\begin{abstract}
\par The purpose of this paper is twofold; (1) to develop several identities for the Generalized $k$-Pell sequence (including those of Binet, Catalan, Cassini, and d'Ocagne), and (2) to study applications of tridiagonal generating matrices for the $k$-Pell and Generalized $k$-Pell sequences.
\end{abstract}
\section{introduction}
\par Since the discovery of the famous Fibonacci sequence and its seemingly innate connection to the natural world, mathematicians have been fascinated by recursive sequences and their properties. For example, the sequence \[\frac{1}{1}, \frac{3}{2}, \frac{7}{5}, \frac{17}{12},\cdots\]
of best rational approximations of $\sqrt{2}$ gives rise to three such recursive integer sequences: the Pell $\{P_n\}$, Pell-Lucas $\{Q_n\}$, and Modified Pell $\{q_n\}$ recursive sequences. These sequences are defined, respectively, as follows. \begin{align*}P_n &= 2P_{n-1}+P_{n-2};\ P_0 = 0, P_1=1\\ Q_n &= 2Q_{n-1}+Q_{n-2};\ Q_0 = Q_1 = 2\\ q_n &= 2q_{n-1}+q_{n-2};\ q_0=q_1=1\end{align*} 
Note that the Pell and Modified Pell sequences, for $n\geq 1$, yield the denominators and numerators, respectively, of the above sequence of rational approximations.
\par Catarino et al.\ \cite{1}\cite{5}\cite{6} defined and studied generalizations of these sequences through using tools from linear algebra. These generalizations --- parameterized by $k\in\ZZ^+$ --- are known as the $k$-Pell sequence $\{P_{k,n}\}$, $k$-Pell-Lucas sequence $\{Q_{k,n}\}$, and Modified $k$-Pell sequence $\{q_{k,n}\}$. These sequences are defined in the following manner.
\begin{align}\label{int:P}P_{k,n} &= 2P_{k,n-1}+kP_{k,n-2};\ P_{k,0} = 0, P_{k,1}=1\\ \label{int:Q}Q_{k,n} &= 2Q_{k,n-1}+kQ_{k,n-2};\ Q_{k,0} = Q_{k,1} = 2\\ q_{k,n} &= 2q_{k,n-1}+kq_{k,n-2};\ q_{k,0}=q_{k,1}=1\end{align}
This paper considers a particular generalization of the Modified $k$-Pell sequence obtained by relaxing its initial conditions. This sequence, herein referred to as the Generalized $k$-Pell sequence, is defined as follows. \begin{align}\label{int:G}G_{k,n} = 2G_{k,n-1} + kG_{k,n-2},\ G_{k,0} = G_{k,1} = a\in\ZZ^+\end{align}
\par The first objective of this paper is to formulate classical identities and other properties for the Generalized k-Pell sequence; this is the focus of Section 2. The next objective is to study tridiagonal generating matrices for the k-Pell and Generalized k-Pell sequences. Specifically, this matrix approach yields closed forms for both $\{P_{k,n}\}$ and $\{G_{k,n}\}$, given in Section 3. Furthermore, in Section 4, several new identities for the k-Pell and Generalized k-Pell sequences are developed by way of considering the matrices of cofactors for their tridiagonal generating matrices.
\par A table of the first several values for $\{P_{k,n}\}$ and $\{G_{k,n}\}$ can be found in Appendix \ref{App:Table}.

\section{properties of the generalized $k$-Pell sequence}
\par In this section, note that the characteristic equation associated to the Generalized $k$-Pell sequence defined in \eqref{int:G} is $r^2-2r-k=0$, with two distinct roots $r_1 = 1+\sqrt{1+k}$ and $r_2 = 1-\sqrt{1+k}$. 
\begin{theorem}[Binet's Formula]\label{thm:binet} The $n^{th}$ Generalized $k$-Pell number is given by \[G_{k,n}= \frac{ar_1^n+ ar_2^n}{2} = \frac{a(1+\sqrt{1+k})^n+ a(1-\sqrt{1+k})^n}{2}\] where $G_{k,0}=G_{k,1}=a\in\ZZ^+$.
\end{theorem}
\begin{proof} Since the characteristic equation of $\{G_{k,n}\}$ has two roots, the closed form of $\{G_{k,n}\}$ is given by $G_{k,n} = c_1(r_1)^n+c_2(r_2)^n$. Evaluating both sides of this equation for $n=0,1$, it follows that $c_1=c_2=a/2$, which concludes the proof.
\end{proof}
The next lemma, which will be helpful in later proofs, immediately follows as a result of Theorem \ref{thm:binet}.
\begin{lemma}\label{prop:thm5}Let $G_{k,n}$ and $Q_{k,n}$ denote the $n^{th}$ terms of the Generalized $k$-Pell and $k$-Pell-Lucas sequences, respectively. Then\[G_{k,n} = \frac{aQ_{k,n}}{2}\]where $G_{k,0}=G_{k,1}=a\in\ZZ^+$.
\end{lemma}
\begin{proof} From \cite[Proposition 1]{6}, Binet's formula for the $k$-Pell-Lucas sequence defined in \eqref{int:Q} is given by $Q_{k,n} = r_1^n+r_2^n$. Combining Binet's formulae for Generalized $k$-Pell and $k$-Pell-Lucas sequences, the conclusion follows.
\end{proof}
Similarly, there is a way to connect the Generalized $k$-Pell sequence and the $k$-Pell sequence as well, which is the objective of the following lemma.
\begin{lemma}\label{prop:lem}Let $G_{k,n}$ and $P_{k,n}$ denote the $n^{\text{th}}$ terms of the Generalized $k$-Pell and $k$-Pell sequences, respectively. Then \[G_{k,n} = aP_{k,n}+akP_{k,n-1}\] where $G_{k,0}=G_{k,1} = a\in\ZZ^+$
\end{lemma}
\begin{proof}
Let $a,k\in\ZZ^+$ be arbitrary, and proceed by induction on $n$. First, let $n = 1$, so that $aP_{k,1}+kaP_{k,0} = a(1)+ka(0) = a = G_{k,1}$. Now, fix $1 < n' \in\ZZ^+$ and suppose that $G_{k,n} = aP_{k,n}+akP_{k,n-1}$ holds for all $n<n'$. Using this hypothesis and the definition of the $k$-Pell and Generalized $k$-Pell sequences, the following equalities hold. \begin{align*}
G_{k,n'} &= 2G_{k,n'-1}+kG_{k,n'-2}\\
&= 2(aP_{k,n'-1}+akP_{k,n'-2})+k(aP_{k,n'-2}+akP_{k,n'-3})\\
&= 2aP_{k,n'-1}+2akP_{k,n'-2}+akP_{k,n'-2}+ak^2P_{k,n'-3}\\
&= a(2P_{k,n'-1}+kP_{k,n'-2})+ak(2P_{k,n'-2}+kP_{k,n'-3})\\
&= aP_{k,n'}+akP_{k,n'-1}
\end{align*}
The result follows by mathematical induction.
\end{proof}
\begin{theorem}[Catalan's identity]\label{prop:cat} Let $G_{k,n}$ denote the $n^{\text{th}}$ term of the Generalized $k$-Pell sequence where $G_{k,0} = G_{k,1} = a\in\ZZ^+$. Then \[G_{k,n-r}G_{k,n+r}-G_{k,n}^2 = (-k)^{n-r}(G_{k,r}^2-a^2(-k)^r)\] for each $r\in\ZZ^+$.
\end{theorem}
\begin{proof} As a result of Binet's formula for the Generalized $k$-Pell and $k$-Pell-Lucas sequences, Lemma \ref{prop:thm5}, the fact that $r_1r_2 = (1+\sqrt{1+k})(1-\sqrt{1+k})=-k$, and Catalan's identity for the $k$-Pell-Lucas sequence \cite[Proposition 2]{6}, the following equalities hold. \begin{align*}
G_{k,n-r}G_{k,n+r}-G_{k,n}^2 &= \lp\frac{ar_1^{n-r}+ ar_2^{n-r}}{2}\rp\lp\frac{ar_1^{n+r}+ ar_2^{n+r}}{2}\rp - \lp\frac{ar_1^{n}+ ar_2^{n}}{2}\rp^2\\
&= \textstyle{\frac{a^2}{4}}\lc\lp r_1^{n-r}+ r_2^{n-r}\rp\lp r_1^{n+r}+ r_2^{n+r}\rp - \lp r_1^{n}+ r_2^{n}\rp^2\rc\\
&= \lp\textstyle{\frac{a^2}{4}}\rp(-k)^{n-r}(Q_{k,n}^2-4(-k)^r)\\
&= (-k)^{n-r}(G_{k,r}^2-a^2(-k)^r).
\end{align*}
\end{proof}
By letting $r=1$ and applying the initial conditions of the Generalized $k$-Pell sequence to Theorem \ref{prop:cat}, Cassini's identity for $\{G_{k,n}\}$ is obtained.
\begin{theorem}[Cassini's identity]Let $G_{k,n}$ denote the $n^{\text{th}}$ term of the Generalized $k$-Pell sequence where $G_{k,0} = G_{k,1} = a\in\ZZ^+$. Then\[G_{k,n-1}G_{k,n+1}-G_{k,n}^2 = a^2(-k)^{n-1}(1+k).\] 
\end{theorem}
\begin{proof}
Let $r=1$ and apply the initial conditions of the Generalized $k$-Pell sequence to Theorem \ref{prop:cat}.
\end{proof}
The last objective of this section is to formulate d'Ocagne's identity for the Generalized $k$-Pell sequence. The proof of this identity will be quite simple with the use of Lemma \ref{prop:thm5}.
\begin{theorem}[d'Ocagne's identity]\label{prop:doc} Let $G_{k,n}$ denote the $n^{\text{th}}$ term of the Generalized $k$-Pell sequence where $G_{k,0} = G_{k,1} = a\in\ZZ^+$. If $m>n$, then \[
G_{k,m}G_{k,n+1}-G_{k,m+1}G_{k,n}= a\lp-1\rp^nk^n\sqrt{1+k}\lp G_{k,m-n}-a\lp1+\sqrt{1+k}\rp^{m-n}\rp.\] 
\end{theorem}
\begin{proof} Again using Binet's formula, \cite[Proposition 4]{6}, and Lemma \ref{prop:thm5}, it follows that \begin{align*}
G_{k,m}G_{k,n+1}-G_{k,m+1}G_{k,n}&=  \lp\frac{ar_1^{m}+ ar_2^{m}}{2}\rp\lp\frac{ar_1^{n+1}+ ar_2^{n+1}}{2}\rp\\&\,\,\,\,\, - \lp\frac{ar_1^{m+1}+ ar_2^{m+1}}{2}\rp\lp\frac{ar_1^{n}+ ar_2^{n}}{2}\rp\\
&= \ts{\frac{a^2}{4}}\lp r_1^{m}+ r_2^{m}\rp\lp r_1^{n+1}+ r_2^{n+1} \rp\\&\,\,\,\,\, - \ts{\frac{a^2}{4}}\lp r_1^{m+1}+ r_2^{m+1}\rp\lp r_1^{n}+ r_2^{n}\rp\\
&= \lp\ts{\frac{a^2}{4}}\rp\lp-1\rp^nk^n2\sqrt{1+k}\lp Q_{k,m-n}-2\lp1+\sqrt{1+k}\rp^{m-n}\rp\\
&= a\lp-1\rp^nk^n\sqrt{1+k}\lp G_{k,m-n}-a\lp1+\sqrt{1+k}\rp^{m-n}\rp.
\end{align*}
\end{proof}

\section{generating matrices}
\par Let $k,n\in\ZZ^+$ and define \begin{align}
\label{gen:P}P_n(k) &= \begin{bmatrix}
2 & k & 0 & 0 & \cdots & 0 & 0 & 0\\
-1 & 2 & k & 0 & \cdots & 0 & 0 & 0\\
0 & -1 & 2 & k & \cdots & 0 & 0 & 0\\
\vdots & \vdots & \vdots & \vdots & \ddots & \vdots & \vdots & \vdots\\
0 & 0 & 0 & 0 & \cdots & -1 & 2 & k\\
0 & 0 & 0 & 0 & \cdots & 0 & -1 & 2\\
\end{bmatrix}_{n\times n}\\ \label{gen:qnk}Q_n(k) &= \begin{bmatrix}
2k+4 & 2k & 0 & 0 & \cdots & 0 & 0 & 0\\
-1 & 2 & k & 0 & \cdots & 0 & 0 & 0\\
0 & -1 & 2 & k & \cdots & 0 & 0 & 0\\
\vdots & \vdots & \vdots & \vdots & \ddots & \vdots & \vdots & \vdots\\
0 & 0 & 0 & 0 & \cdots & -1 & 2 & k\\
0 & 0 & 0 & 0 & \cdots & 0 & -1 & 2\\
\end{bmatrix}_{n\times n}\\ \label{gen:unk}q_n(k) &= \begin{bmatrix}
k+2 & k & 0 & 0 & \cdots & 0 & 0 & 0\\
-1 & 2 & k & 0 & \cdots & 0 & 0 & 0\\
0 & -1 & 2 & k & \cdots & 0 & 0 & 0\\
\vdots & \vdots & \vdots & \vdots & \ddots & \vdots & \vdots & \vdots\\
0 & 0 & 0 & 0 & \cdots & -1 & 2 & k\\
0 & 0 & 0 & 0 & \cdots & 0 & -1& 2\\
\end{bmatrix}_{n\times n}\end{align}
It was shown in \cite{1} that these matrices generate the terms of the $k$-Pell, $k$-Pell-Lucas, and Modified $k$-Pell sequences, respectively, in the sense that  \begin{align*}
\la P_{n}(k)\ra &= P_{k,n+1}\\
\la Q_{n}(k)\ra &= Q_{k,n+1}\\
\la q_{n}(k)\ra &= q_{k,n+1}\end{align*}
for all $k,n\in\ZZ^+$. 
\par Using methods similar to those found in \cite{1}, one can adjust the Modified $k$-Pell generating matrix $q_n(k)$ to define a generating matrix for the Generalized $k$-Pell sequence.
\begin{theorem} The matrix \begin{align}\label{gen:G}G_n(k) = \begin{bmatrix}
ak+2a & ak & 0 & 0 & \cdots & 0 & 0 & 0\\
-1 & 2 & k & 0 & \cdots & 0 & 0 & 0\\
0 & -1 & 2 & k & \cdots & 0 & 0 & 0\\
\vdots & \vdots & \vdots & \vdots & \ddots & \vdots & \vdots & \vdots\\
0 & 0 & 0 & 0 & \cdots & -1 & 2 & k\\
0 & 0 & 0 & 0 & \cdots & 0 & -1 & 2\\
\end{bmatrix}_{n\times n}\end{align} is a generating matrix for the Generalized $k$-Pell sequence, $\{G_{k,n}\}$, with initial conditions $G_{k,0} = G_{k,1} = a\in\ZZ^+$ and $a\ne 0$. That is, $\la G_n(k)\ra = G_{k,n+1}\ \forall k,n\in\ZZ^+$.\end{theorem}
\begin{proof} Let $a,k\in\ZZ^+$ be arbitrary and let $n = 1$, so that $G_1(k) = \begin{bmatrix} ak+2a \end{bmatrix}$. It is obvious that $\la G_1(k) \ra =  G_{k,2}$. Now, fix $1<m\in \ZZ^+$ and suppose that $\la G_{n}(k) \ra = G_{k,n+1}$ for all $n\in\ZZ^+$ such that $n < m$. $G_{m}(k)$ is of the form \[G_{m}(k) =\begin{bmatrix}
ak+2a & ak & 0 & 0 & \cdots & 0 & 0 & 0\\
-1 & 2 & k & 0 & \cdots & 0 & 0 & 0\\
0 & -1 & 2 & k & \cdots & 0 & 0 & 0\\
\vdots & \vdots & \vdots & \vdots & \ddots & \vdots & \vdots & \vdots\\
0 & 0 & 0 & 0 & \cdots & -1 & 2 & k\\
0 & 0 & 0 & 0 & \cdots & 0 & -1 & 2\\
\end{bmatrix}_{m\times m}\] Expanding along the last column of this matrix, it follows that $\la G_{m}(k)\ra = 2\la G_{m-1}(k)\ra + k \la G_{m-2}(k)\ra = 2G_{k,m} + kG_{k,m-1} = G_{k,m+1}$. Therefore, since $a,k\in\ZZ^+$ were set to be arbitrary, the result follows by induction.
\end{proof}
Consider the first few values of $\la G_n(k)\ra$, displayed below.\begin{align*}
\la G_1(k)\ra &= ak+2a\\
\la G_2(k)\ra &= \la \begin{matrix}
ak+2a & ak \\
-1 & 2 & \\
\end{matrix}\ra\\
&= 2(ak+2a) + ak\\
\la G_3(k)\ra &=\la \begin{matrix}
ak+2a& ak  & 0\\
-1 & 2 &  k\\
0 & -1 & 2 \\
\end{matrix}\ra\\
&= 2\left[2(ak+2a) + ak\right] + k(ak+2a)\\
&= 2^2(ak+2a)+2ak+k(ak+2a)
\end{align*}
Interestingly, there is a general behavior for the value of $\la G_n(k)\ra$ in terms of $n$, which is the subject of the following theorem. Assume, for sake of clarity, that ${-1\choose 0} = 0$.
\begin{theorem} Let $G_{k,n}$ denote the $n^{\text{th}}$ term of the Generalized $k$-Pell sequence with initial conditions $G_{k,0} = G_{k,1} = a\in\ZZ^+$. Then \[ G_{k,n+1}= \left\{ \begin{array}{ll}
\displaystyle{\sum\limits_{i=1}^m \sum\limits_{j=0}^1} \textstyle{{m-2+i+j \choose m-i}} a^{1-j}k^{m+1-i-j}2^{2i+j-2}(ak+2a)^j &\text{ for } n =2m, \text{ and}\\
 \displaystyle{\sum\limits_{i=1}^m \sum\limits_{j=0}^1} \textstyle{{m-3+i+j \choose m-i}} a^{1-j}k^{m+1-i-j}2^{2i+j-3}(ak+2a)^j & \text{ for } n =2m-1
\end{array}\right.\]\end{theorem}
\begin{proof}
Let $k\in\ZZ^+$ be arbitrary, and proceed by mathematical induction on $m$. First, let $m=1$, and suppose that $n=2m=2$. It follows that \begin{align*}
G_{k,3} &= 3ka+4a\\
&= \ts{{0\choose 0}}ak + \ts{{1\choose 0}}2(ak+2a)\\
&= \sum\limits_{i=1}^1 \sum\limits_{j=0}^1 \ts{{1-2+i+j \choose 1-i}} a^{1-j}k^{1+1-i-j}2^{2i+j-2}(ak+2a)^j
\end{align*} Now, assume that $n = 2m-1 = 1$, so that\begin{align*}
G_{k,2} &= ka +2a\\
&= 0+ \ts{{0\choose0}}(ak+2a)^1\\
&=\sum\limits_{i=1}^1 \sum\limits_{j=0}^1 \ts{{1-3+i+j \choose 1-i}} a^{1-j}k^{1+1-i-j}2^{2i+j-3}(ak+2a)^j
\end{align*} Now suppose that the statement holds for $m=m'\in\ZZ^+$, with the goal of showing that it also holds for $m'+1\in\ZZ^+$. To this end, let $m = m'+1$, and first consider $n=2m-1 =2m'+1$. It follows that \begin{align*}
G_{k,2m'+1} &= 2G_{k,2m'}+kG_{k,2m'-1}\\
&= 2\sum\limits_{i=1}^{m'} \sum\limits_{j=0}^1 \ts{{m'-2+i+j \choose m'-i}} a^{1-j}k^{m'+1-i-j}2^{2i+j-2}(ak+2a)^j\\&\,\,\,\,\, + k\sum\limits_{i=1}^{m'} \sum\limits_{j=0}^1 \ts{{m'-3+i+j \choose m'-i}} a^{1-j}k^{m'+1-i-j}2^{2i+j-3}(ak+2a)^j\\
&= \sum\limits_{i=1}^{m'} \sum\limits_{j=0}^1 \ts{{m'-2+i+j \choose m'-i}} a^{1-j}k^{m'+1-i-j}2^{2i+j-1}(ak+2a)^j\\&\,\,\,\,\, + \sum\limits_{i=1}^{m'} \sum\limits_{j=0}^1 \ts{{m'-3+i+j \choose m'-i}} a^{1-j}k^{m'+2-i-j}2^{2i+j-3}(ak+2a)^j\\
&=\ts{{m'-1\choose m'-1}}ak^{m'}2+\ts{{m'\choose m'-1}}k^{m'-1}2^2(ak+2a)+\cdots+\ts{{2m'-3}\choose{1}}ak^{2}2^{2m'-3}\\&\,\,\,\,\, +\ts{{2m'-2}\choose{1}}k2^{2m'-2}(ak+2a)+ \ts{{2m'-2}\choose{0}}ak2^{2m'-1} + \ts{{2m'-1}\choose{0}}2^{2m'}(ak+2a) \\&\,\,\,\,\, + \ts{{m'-1}\choose{m'-1}}k^{m'}(ak+2a) +\ts{{m'-1\choose m'-2}}ak^{m'}2+\ts{{m'\choose m'-2}}k^{m'-1}2^2(ak+2a) + \cdots \\ &\,\,\,\,\, + \ts{{2m'-3}\choose{0}}ak^22^{2m'-3} + \ts{{2m'-2}\choose{0}}k2^{2m'-2}(ak+2a)\\
&= \ts{{m'}\choose{m'}}k^{m'}(ak+2a) + \ts{{m'\choose m'-1}}ak^{m'}2+\ts{{m'+1\choose m'-1}}k^{m'-1}2^2(ak+2a)+\cdots
\\&\,\,\,\,\, +\ts{{2m'-2}\choose{1}}ak^{2}2^{2m'-3} +\ts{{2m'-1}\choose{1}}k2^{2m'-2}(ak+2a)+ \ts{{2m'-1}\choose{0}}ak2^{2m'-1} \\&\,\,\,\,\, +\ts{{2m'}\choose{0}}2^{2m'}(ak+2a)\\
&= \sum\limits_{i=1}^{m'+1} \sum\limits_{j=0}^1 \textstyle{{m'-2+i+j \choose m'-i+1}} a^{1-j}k^{m'+2-i-j}2^{2i+j-3}(ak+2a)^j.
\end{align*}
Now, consider $n = 2m = 2m'+2$. Thus,\begin{align*}
G_{k,2m'+2} &= 2G_{k,2m'+1}+kG_{k,2m'} \\
&= 2\sum\limits_{i=1}^{m'+1} \sum\limits_{j=0}^1 \textstyle{{m'-2+i+j \choose m'-i+1}} a^{1-j}k^{m'+2-i-j}2^{2i+j-3}(ak+2a)^j\\
&\,\,\,\,\, +k\sum\limits_{i=1}^{m'} \sum\limits_{j=0}^1 \textstyle{{m'-2+i+j \choose m'-i}} a^{1-j}k^{m'+1-i-j}2^{2i+j-2}(ak+2a)^j \\
&= \sum\limits_{i=m'+1}^{m'+1} \sum\limits_{j=0}^1 \textstyle{{m'-2+i+j \choose m'-i+1}} a^{1-j}k^{m'+2-i-j}2^{2i+j-2}(ak+2a)^j \\ &\,\,\,\,\, +\sum\limits_{i=1}^{m'} \sum\limits_{j=0}^1 \textstyle{{m'-2+i+j \choose m'-i}} a^{1-j}k^{m'+2-i-j}2^{2i+j-2}(ak+2a)^j \\
&\,\,\,\,\, +\sum\limits_{i=1}^{m'} \sum\limits_{j=0}^1 \textstyle{{m'-2+i+j \choose m'-i+1}} a^{1-j}k^{m'+2-i-j}2^{2i+j-2}(ak+2a)^j\\
&= \sum\limits_{i=1}^{m'} \sum\limits_{j=0}^1 \textstyle{{m'-1+i+j \choose m'-i+1}} a^{1-j}k^{m'+2-i-j}2^{2i+j-2}(ak+2a)^j \\ &\,\,\,\,\, + \sum\limits_{i=m'+1}^{m'+1} \sum\limits_{j=0}^1 \textstyle{{m'-1+i+j \choose m'-i+1}} a^{1-j}k^{m'+2-i-j}2^{2i+j-2}(ak+2a)^j \\
&= \sum\limits_{i=1}^{m'+1} \sum\limits_{j=0}^1 \textstyle{{m'-1+i+j \choose m'-i+1}} a^{1-j}k^{m'+2-i-j}2^{2i+j-2}(ak+2a)^j.
\end{align*}
The result follows by mathematical induction on $m$.
\end{proof}
Now consider the first few values of $\la P_n(k)\ra$ in terms of $n$, displayed below. 
\begin{align*}
\la P_1(k)\ra &= 2\\
\la P_2(k)\ra &= \la \begin{matrix}
2 & k\\
-1 & 2
\end{matrix} \ra\\
&= 4+k\\
\la P_3 \ra &= \la \begin{matrix}
2 &k & 0\\
-1 & 2 & k\\
0 & -1 & 2
\end{matrix}\ra\\
&= 2(4+k) + k(2)\\
&= 4k + 8
\end{align*}
The following theorem provides a formula for the determinant of $P_n(k)$.
\begin{theorem}\label{ten:closeP}  Let $P_{k,n}$ denote the $n^{\text{th}}$ term of the $k$-Pell sequence with initial conditions $P_{k,0} =0, P_{k,1} = 1$. If $n\geq 2$, then \[ P_{k,n+1}= \sum\limits_{i=0}^{\lfloor\frac{n}{2}\rfloor}  {n-i \choose i} k^i2^{n-2i} \]
\end{theorem}
\begin{proof} Let $k\in\ZZ^+$ be arbitrary, and proceed by induction on $n$. If $n = 2$, then $\ds P_{k,3}= k+4 =  \sum\limits_{i=0}^1  \ts{{2-i \choose i}} k^i2^{2-2i}$. Fix $2<n'\in\ZZ^+$ and assume that the statement holds for all $n<n'$. First, suppose $n'$ is even, so that $n' = 2m$ for some $m\in\ZZ^+$. By the induction hypothesis and the definition of $\{P_{k,n}\}$ from \eqref{int:P}, the following equalities hold. \begin{align*}P_{k,n'} &= 2P_{k,n'-1}+kP_{k,n'-2}\\
&= 2\lp \sum\limits_{i=0}^{m-1}  \ts{{n'-1-i \choose i}} k^i2^{n'-1-2i}\rp + k\lp \sum\limits_{i=0}^{m-1}  \ts{{n'-2-i \choose i}} k^i2^{n'-2-2i}\rp \\
&=  \sum\limits_{i=0}^{m-1}  \ts{{n'-1-i \choose i}} k^i2^{n'-2i}+ \sum\limits_{i=0}^{m-1}  \ts{{n'-2-i \choose i}} k^{i+1}2^{n'-2-2i} \\
&= \ts{{n'-1 \choose 0}} k^02^{n'} +  \ts{{n'-2 \choose 1}} k^12^{n'-2} +  \ts{{n'-3 \choose 2}} k^22^{n'-4}+\cdots \ts{{n'-m \choose m-1}}k^{m-1}2^{n'-2m+2}\\&\,\,\,\,\,+\ts{{n'-2 \choose 0}} k^{1}2^{n'-2}+ \ts{{n'-3 \choose 1}} k^{2}2^{n'-4}+\cdots + \ts{{n'-m \choose m-2}} k^{m-1}2^{n'-2m+2} \\&\,\,\,\,\,+  \ts{{n'-m-1 \choose m-1}} k^{m}2^{n'-2m}\\
&= \ts{{n' \choose 0}} k^02^{n'}+  \ts{{n'-1 \choose 1}} k^12^{n'-2} +  \ts{{n'-2 \choose 2}} k^22^{n'-4}+\cdots \ts{{n'-m+1 \choose m-1}}k^{m-1}2^{n'-2m+2}\\&\,\,\,\,\,+  \ts{{n'-m \choose m}} k^{m}2^{n'-2m}\\
&= \sum\limits_{i=0}^m  \ts{{n'-i \choose i}} k^i2^{n'-2i}
\end{align*} The case where $n$ is odd follows similarly, and the result is obtained.
\end{proof}
The last closed form for $\{P_{k,n}\}$ in this paper is a result of considering the eigenvalues of $P_n(k)$, defined in \eqref{gen:P}. The following theorem, the proof of which can be found in \cite{3}, will provide a necessary tool for calculating these eigenvalues.
\begin{theorem}[\cite{3}]\label{eig:thm} Let $T_n$ be an $n\times n$ tridiagonal matrix of the form \[ T_n = \begin{bmatrix}
a & b & 0 & 0 & \cdots & 0 & 0 & 0\\
c & a & b & 0 & \cdots & 0 & 0 & 0\\
0 & c & a & b & \cdots & 0 & 0 & 0\\
\vdots & \vdots & \vdots & \vdots & \ddots & \vdots & \vdots & \vdots\\
0 & 0 & 0 & 0 & \cdots & c & a & b\\
0 & 0 & 0 & 0 & \cdots & 0 & c & a\\
\end{bmatrix}_{n\times n}\]
Then $T_n$ has eigenvalues $\ds\lambda_r = a + 2\sqrt{bc}\cos{\lp \frac{r\pi}{n+1}\rp} \text{ for } r = 1,2,\dots,n$.
\end{theorem} Applying Theorem \ref{eig:thm} to the matrix $P_n(k)$ gives rise to the following theorem.
\begin{theorem}\label{eig:eig} Let $P_{k,n}$ denote the $n^{\text{th}}$ term of the $k$-Pell sequence. Then \[ P_{k,n+1}=\prod\limits_{r=1}^n \lp 2+i\sqrt{k}\cos{\lp \frac{r\pi}{n+1}\rp} \rp  \]
\end{theorem}
\begin{proof} Consider the matrix $P_n(k)$ defined \eqref{gen:P}. Note that $P_n(k)$ is a tridiagonal Toeplitz matrix. Thus, it follows from Theorem \ref{eig:thm} that the eigenvalues of $P_n(k)$ are $\ds \lambda _r = 2+i\sqrt{k}\cos{\lp \frac{r\pi}{n+1}\rp} \text{ for } r = 1,2,\dots,n$. Since $\la P_n(k)\ra = P_{k,n+1}$, the conclusion follows.
\end{proof}
\par Similarly to what was noted in \cite[Section 4.1]{2}, if $n$ is odd, then the matrix $P_n(k)$ has exactly one real eigenvalue, $\lambda_r = 2$, where $r = \frac{n+1}{2}$. If $n$ is even, $P_n(k)$ has no real eigenvalues. Regardless, the product in Theorem \ref{eig:eig} will always be an integer, an interesting fact due to the ``complexity'' of the statement.

\section{matrix of cofactors}
\begin{theorem}[Usmani's Formula, \cite{4}] Consider the nonsingular, tridiagonal matrix $U_n$ defined by \[U_n = \begin{bmatrix}
a_1 & b_1 & 0 & 0 & \cdots & 0 & 0 & 0\\
c_1 & a_2 & b_2 & 0 & \cdots & 0 & 0 & 0\\
0 & c_2 & a_3 & b_3 & \cdots & 0 & 0 & 0\\
\vdots & \vdots & \vdots & \vdots & \ddots & \vdots & \vdots & \vdots\\
0 & 0 & 0 & 0 & \cdots & c_{n-2} & a_{n-1} & b_{n-1}\\
0 & 0 & 0 & 0 & \cdots & 0 & c_{n-1} & b_n\\
\end{bmatrix}\]
where $a_i,b_i,c_i\in\RR$. Then $U_n^{-1}$ is given by \[[U_n^{-1}]_{i,j} = \left\{ \begin{array}{ll}
         (-1)^{i+j}\frac{1}{\theta_n}b_i\cdots b_{j-1}\theta_{i-1}\phi_{j+1} &\text{if } i < j,\\
                  \frac{1}{\theta_n}\theta_{i-1}\phi_{i+1} &\text{if } i=j, \text{ and}\\
         (-1)^{i+j}\frac{1}{\theta_n}c_j\cdots c_{i-1}\theta_{j-1}\phi_{i+1} & \text{if } i > j\end{array} \right. \]
where \begin{align*}
\theta_i &= a_i\theta_{i-1}-b_{i-1}c_{i-1}\theta_{i-2};\quad \theta_0 = 1, \theta_1 = a_1 \\
\phi_i &= a_i\phi_{i+1}-b_{i}c_{i}\phi_{i+2};\quad \phi_{n+1} = 1, \phi_n = a_n
\end{align*}
for all $1\leq i \leq n$.
\end{theorem}
\par Applying Usmani's Formula to $P_n(k)$ and $G_n(k)$ yields the Theorems \ref{cof:usmaniP} and \ref{cof:usmaniG}. 
\begin{theorem}\label{cof:usmaniP} Let $P_n(k)$ be the matrix defined in \eqref{gen:P} and let $P_{k,n}$ denote the $n^{\text{th}}$ term in the $k$-Pell sequence. Then $(P_n(k))^{-1}$ is given by \[[(P_n(k))^{-1}]_{i,j} = \left\{ \begin{array}{ll}
         (-1)^{i+j}\frac{1}{P_{k,n+1}}k^{j-i}P_{k,i}P_{k,n-j+1} &\text{if } i \leq j,\text{ and}\\
         \frac{1}{P_{k,n+1}}P_{k,j}P_{k,n-i+1} & \text{if } i > j.\end{array} \right. \]
\end{theorem}
\begin{proof}First, note that $P_n(k)$ is nonsingular and tridiagonal. Thus, \[[(P_n(k))^{-1}]_{i,j} =  \left\{ \begin{array}{ll}
         (-1)^{i+j}\frac{1}{\theta_n}k^{j-i}\theta_{i-1}\phi_{j+1} &\text{if } i < j, \text{ and}\\
         \frac{1}{\theta_n}\theta_{j-1}\phi_{i+1} & \text{if } i \geq j\end{array} \right. \] where\begin{align*}
\phi_{n+1} &= 1\\
\phi_n &= 2\\
\phi_{n-1} &= 2\phi_n +k\phi_{n+1} = 4 + k\\
\phi_{n-2} &= 2\phi_{n-1} + k\phi_n = 8 + 4k\\
\phi_{n-3} &= 2\phi_{n-2} + k\phi_{n-1} = 16 + 12k+k^2\\
&\,\,\,\,\vdots\\
\phi_j &= 2\phi_{j+2} + k\phi_{j+1} = P_{k,n-j+2}
\end{align*}
and
\begin{align*}
\theta_0 &= 1\\
\theta_1 &= 2\\
\theta_2 &= 2(2) +k(1) = 4+k\\
\theta_3 &= 2(4+k) + k(2) = 8+4k\\
\theta_4 &= 2(8+4k) +k(4+k) = 16 + 12k + k^2\\
&\,\,\,\,\vdots\\
\theta_i &= 2\theta_{i-1} + k\theta_{i-2} = P_{k,i+1}
\end{align*}
Therefore, Usmani's Formula for $(P_n(k))^{-1}$ becomes \[[(P_n(k))^{-1}]_{i,j} = \left\{ \begin{array}{ll}
         (-1)^{i+j}\frac{1}{P_{k,n+1}}k^{j-i}P_{k,i}P_{k,n-j+1} &\text{if } i\leq j, \text{ and}\\
         \frac{1}{P_{k,n+1}}P_{k,j}P_{k,n-i+1} & \text{if } i > j.\end{array} \right. \]
\end{proof}
\par Note that the entries of the matrix of cofactors for $P_n(k)$ follow directly from Theorem \ref{cof:usmaniP}. \[ c_{i,j}(P_n(k))= \left\{ \begin{array}{ll}
         (-1)^{i+j}k^{i-j}P_{k,j}P_{k,n-i+1} &\text{if } i \geq j, \text{ and}\\
         P_{k,i}P_{k,n-j+1} & \text{if } i < j.\end{array} \right. \]
The matrix of cofactors for $P_n(k)$ for $n\geq 2$ is given by 
\begin{align}\label{cof:C}C_{n}(k) = \scriptsize \begin{bmatrix}
P_{k,n} & P_{k,n-1} & P_{k,n-2} &  \cdots  & P_{k,2} & P_{k,1} \\
-kP_{k,n-1} & P_{k,2}P_{k,n-1} & P_{k,2}P_{k,n-2} & \cdots  & P_{k,2}P_{k,2} & P_{k,2}\\
k^2P_{k,n-1} & -kP_{k,2}P_{k,n-2} & P_{k,3}P_{k,n-2} & \cdots & P_{k,3}P_{k,2} & P_{k,3}\\
\vdots & \vdots & \vdots & \ddots & \vdots & \vdots\\
(-1)^{n}k^{n-2}P_{k,2} & (-1)^{n+1}k^{n-3}P_{k,2}P_{k,2} & (-1)^{n+2}k^{n-4}P_{k,3}P_{k,2} & \cdots& P_{k,n-1}P_{k,2} & P_{k,n-1}\\
(-1)^{n+1}k^{n-1}P_{k,1} & (-1)^{n+2}k^{n-2}P_{k,2} & (-1)^{n+3}k^{n-3}P_{k,3} & \cdots& (-1)^{2n-1}kP_{k,n-1} & P_{k,n}
\end{bmatrix}\end{align}
By examining the determinant of this matrix, the following lemma arises almost instantly.
\begin{proposition}\label{cof:Plem} Let $P_n(k)$ be the matrix defined in \eqref{gen:P} and let $C_{n}(k)$ be the $n\times n$ matrix of cofactors for $P_n(k)$. Then $\la C_{n}(k) \ra = P_{k,n+1}^{n-1}$.
\end{proposition}
\begin{proof} Assuming the hypotheses of the lemma, it follows that \[ (P_n(k))^{-1} =  \frac{1}{\la P_n(k)\ra}\text{Adj}(P_n(k)) = \frac{1}{\la P_n(k)\ra}C_{n}(k)^T.\] Multiplying on the left by $\la P_n(k)\ra P_n(k)$ yields $\ds\la P_n(k)\ra I_n  = P_n(k)C_{n}(k)^T$. By properties of the determinant and since $\la P_n(k)\ra = P_{k,n+1}$, the following implications hold. \begin{align*} 
\la \la P_n(k)\ra I_n \ra = \la P_n(k)C_{n}(k)^T\ra &\implies 
\la P_n(k)\ra^n = \la P_n(k)\ra\la C_{n}(k)\ra \\ &\implies 
P_{k,n+1}^n = P_{k,n+1}\la C_{n}(k)\ra \\ &\implies
P_{k,n+1}^{n-1} = \la C_{n}(k)\ra
\end{align*}
\end{proof}
By using Proposition \ref{cof:Plem} and the fact that $\ds P_{k,n} = \frac{P_{k,n+1}-kP_{k,n-1}}{2}$ (by the definition of the $k$-Pell sequence), an interesting result can be derived from the determinant of $C_{n}(k)$. This is the subject of Theorem \ref{cof:Pthm}, the proof of which will employ the use of the following two lemmas. Note that Lemma \ref{cof:lemlem} was originally conjectured and proven in \cite[Proposition 5]{5}.
\begin{lemma}\label{cof:lemlem} Let $P_{k,i}$ denote the $i^{\text{th}}$ term of the $k$-Pell sequence and let $m,n\in\ZZ^+$. Then\begin{enumerate}
\item\label{cof:lemlem1}$P_{k,n+m} = kP_{k,n-1}P_{k,m}+P_{k,n}P_{k,m+1}, \text{ and}$
\item\label{cof:lemlem2}$2P_{k,n+m} = P_{k,n+1}P_{k,m+1}-k^2P_{k,m-1}P_{k,n-1}$
\end{enumerate}
\end{lemma}
\begin{proof} To prove \eqref{cof:lemlem1}, let $k,n\in\ZZ^+$ be arbitrary and induct on $m$ (see \cite[Proposition 5]{5} for details). To prove \eqref{cof:lemlem2}, let $k,n\in\ZZ^+$ be arbitrary, and proceed by induction on $m$. If $m=1$, then by the definition of the $k$-Pell sequence, $2P_{k,n+1} = 2P_{k,n+1}-(0)k^2P_{k,n-1} = P_{k,n+1}P_{k,2}-k^2P_{k,0}P_{k,n-1}$. Let $1<m'\in\ZZ^+$ and assume that $2P_{k,n+m} = P_{k,n+1}P_{k,m+1}-k^2P_{k,m-1}P_{k,n-1}$ holds for all $m<m'$. By this hypothesis, the following equalities hold. \begin{align*}
2P_{k,n+m'} &= 2\lp  2P_{k,n+m'-1}+kP_{k,n+m'-2} \rp \\ 
&= 2(2P_{k,n+m'-1})+k(2P_{k,n+m'-2})\\
&= 2\lp P_{k,n+1}P_{k,m'}-k^2P_{k,m'-2}P_{k,n-1}\rp +k\lp P_{k,n+1}P_{k,m'-1}-k^2P_{k,m'-3}P_{k,n-1}\rp\\
&= 2P_{k,n+1}P_{k,m'}+kP_{k,n+1}P_{k,m'-1}-2k^2P_{k,m'-2}P_{k,n-1}-k^3P_{k,m'-3}P_{k,n-1}\\
&= P_{k,n+1}\lp 2P_{k,m'}+kP_{k,m'-1}\rp -k^2P_{k,n-1}\lp 2P_{k,m'-2}+kP_{k,m'-3}\rp \\ 
&= P_{k,n+1}P_{k,m'+1}-k^2P_{k,m'-1}P_{k,n-1}
\end{align*} Since $n\in\ZZ^+$ was arbitrary, the result follows by mathematical induction.
\end{proof}
Now, by examining the determinant of $C_{n}(k)$ as defined in \eqref{cof:C} for $n = 1,2,3$, the following relations hold true by Lemma \ref{cof:Plem}. \begin{align*}
\la C_1(k)\ra &= P_{k,1} = 1\\
\la C_2(k)\ra &= \la\begin{matrix}
P_{k,2} & P_{k,1}\\
-kP_{k,1} & P_{k,2}\\
\end{matrix}\ra = P_{k,3}\\
&\implies P_{k,2}^2+kP_{k,1}^2 = P_{k,3}\\
\la C_3(k)\ra &= \la\begin{matrix}
P_{k,3} & P_{k,2} & P_{k,1}\\
-kP_{k,2} & P_{k,2}^2 & P_{k,2}\\
k^2P_{k,1} & -kP_{k,2} & P_{k,3}\\
\end{matrix}\ra = P_{k,4}^2\\
&\implies P_{k,3}\lp P_{k,2}^2P_{k,3}+kP_{k,2}^2\rp + P_{k,2}\lp kP_{k,2}P_{k,3}+k^2P_{k,2}P_{k,1}\rp = P_{k,4}^2\\
&\implies P_{k,2}^2\lp P_{k,3}^2+2kP_{k,3}P_{k,1} + k^2P_{k,1}^2\rp = P_{k,4}^2\\
&\implies \lp \frac{P_{k,3}-kP_{k,1}}{2} \rp^2\lp P_{k,3}+ kP_{k,1}\rp^2 = P_{k,4}^2\\
&\implies \ts{\frac{1}{4}}\lp P_{k,3}^2 - k^2P_{k,1}^2\rp ^2= P_{k,4}^2\\
&\implies P_{k,3}^2 - k^2P_{k,1}^2 = 2P_{k,4}
\end{align*}
Theorem \ref{cof:Pthm} is the result of the generalization of the determinants of $C_{n-1}(k)$. Note that \eqref{cof:part2} was originally proven in \cite[Proposition 6]{5}.
\begin{theorem}\label{cof:Pthm} Let $P_{k,n}$ denote the $n^{\text{th}}$ term of the $k$-Pell sequence. Then\begin{enumerate}
\item\label{cof:part1}$P_{k,n+1}^2+kP_{k,n}^2 = P_{k,2n+1}, \text{ and}$
\item\label{cof:part2}$P_{k,n+1}^2-k^2P_{k,n-1}^2 = 2P_{k,2n}$
\end{enumerate}\end{theorem}
\begin{proof} To prove \eqref{cof:part1}, let $n = m+1$ and apply Lemma \ref{cof:lemlem}\eqref{cof:lemlem1} to obtain $P_{k,m+1}^2+kP_{k,m}^2 = P_{k,2m+1}$, as desired. To prove \eqref{cof:part2}, let $n = m$ and apply Lemma \ref{cof:lemlem}\eqref{cof:lemlem2}. This yields $P_{k,n+1}^2-k^2P_{k,n-1}^2 = 2P_{k,2n}$, which completes the proof.
\end{proof}
\par The next objective concerns the matrix of cofactors for the matrix $G_n(k)$ as defined in \eqref{gen:G}. Again, using Usmani's formula, the following theorem regarding the inverse of $G_n(k)$ arises.
\begin{theorem}\label{cof:usmaniG} Let $G_n(k)$ be the matrix defined in \eqref{gen:G} and let $G_{k,n}$ and $P_{k,n}$ denote the $n^{\text{th}}$ terms of the Generalized $k$-Pell and $k$-Pell sequences, respectively. The entries of $(G_n(k))^{-1}$ are given by \[[(G_n(k))^{-1}]_{i,j} = \left\{ \begin{array}{ll}
	(-1)^{j+1}\frac{1}{G_{k,n+1}}ak^{j-1}P_{k,n-j+1} &\text{if } 1=i < j,\\
         (-1)^{i+j}\frac{1}{G_{k,n+1}}k^{j-i}G_{k,i}P_{k,n-j+1} &\text{if } 1< i \leq j,\\
         \frac{1}{G_{k,n+1}}P_{k,n-i+1} & \text{if } i \geq j =1, \text{ and}\\
         \frac{1}{G_{k,n+1}}G_{k,j}P_{k,n-i+1} & \text{if } i > j >1 .\end{array} \right. \]
\end{theorem}
\begin{proof} Note that $G_n(k)$ satisfies the hypotheses of Usmani's Formula. Thus, \[[(G_n(k))^{-1}]_{i,j} =  \left\{ \begin{array}{ll}
	(-1)^{i+j}\frac{1}{\theta_n}ak^{j-i}\theta_{0}\phi_{j+1} &\text{if } 1=i < j,\\
         (-1)^{i+j}\frac{1}{\theta_n}k^{j-i}\theta_i\phi_{j+1} &\text{if } 1< i \leq j,\\
         \frac{1}{\theta_n}\phi_{i+1} & \text{if } i \geq j =1, \text{ and}\\
         \frac{1}{\theta_n}\theta_{j}\phi_{i+1} & \text{if } i > j >1 .\end{array} \right. \] where\begin{align*}
\phi_{n+1} &= 1\\
\phi_n &= 2\\
\phi_{n-1} &= 2\phi_n +k\phi_{n+1} = 4 + k\\
\phi_{n-2} &= 2\phi_{n-1} + k\phi_n = 8 + 4k\\
\phi_{n-3} &= 2\phi_{n-2} + k\phi_{n-1} = 16 + 12k+k^2\\
&\,\,\,\,\vdots\\
\phi_j &= 2\phi_{j+2} + k\phi_{j+1} = P_{k,n-j+2}
\end{align*}
and
\begin{align*}
\theta_0 &= 1\\
\theta_1 &= ak+2a = G_{k,2}\\
\theta_2 &= 2(ak+2a) -ak(-1)(1) = 3ak+4a = G_{k,3}\\
\theta_3 &= 2(3ak+4a) -k(-1)(ak+2a) = ak^2+8ak+8a = G_{k,4}\\
&\,\,\,\,\vdots\\
\theta_i &= 2\theta_{i-1} + k\theta_{i-2} = G_{k,i+1}
\end{align*}
Note that the definition of $\{\phi_j\}$ for $G_n(k)$ is the same as it was for $P_n(k)$, due to the fact that $P_n(k)$ and $G_n(k)$ only differ in their first rows and the fact that Usmani's Formula does not account for $\phi_1$. Combining all cases, the result follows by Usmani's formula.
\end{proof}
By the same argument as for the matrix of cofactors for $P_n(k)$, the entries of the matrix of cofactors for $G_n(k)$ are\[c_{i,j}(G_n(k))=\left\{ \begin{array}{ll}
	(-1)^{i+j}ak^{i-j}P_{k,n-i+1} &\text{if } i > j=1,\\
         (-1)^{i+j}k^{i-j}G_{k,j}P_{k,n-i+1} &\text{if } i \geq j>1,\\
        P_{k,n-j+1} & \text{if } 1 = i \leq j, \text{ and}\\
        G_{k,i}P_{k,n-j+1} & \text{if } 1<i < j .\end{array} \right.\]
Thus, the matrix of cofactors for $G_n(k)$ for $n\geq 2$ can be expressed as $D_{n}(k)=$
\[\scriptsize\begin{bmatrix}
P_{k,n} & P_{k,n-1} & P_{k,n-2} &  \cdots  & P_{k,2} & P_{k,1}\\
-akP_{k,n-1} & G_{k,2}P_{k,n-1} & G_{k,2}P_{k,n-2} & \cdots  & G_{k,2}P_{k,2} & G_{k,2}P_{k,1}\\
ak^2P_{k,n-2} & -kG_{k,2}P_{k,n-2} & G_{k,3}P_{k,n-2} & \cdots & G_{k,3}P_{k,2} & G_{k,3}P_{k,1}\\
\vdots & \vdots & \vdots & \ddots & \vdots & \vdots\\
(-1)^{n}ak^{n-2}P_{k,2} & (-1)^{n+1}k^{n-3}G_{k,2}P_{k,2} & (-1)^{n+2}k^{n-4}G_{k,3}P_{k,2} & \cdots& G_{k,n-1}P_{k,2} & G_{k,n-1}P_{k,1}\\
(-1)^{n+1}ak^{n-1}P_{k,1} & (-1)^{n+2}k^{n-2}G_{k,2}P_{k,1} & (-1)^{n+3}k^{n-3}G_{k,3}P_{k,1} & \cdots& (-1)^{2n-1}kG_{k,n-1}P_{k,1} & G_{k,n}P_{k,1}\\
\end{bmatrix}_{n\times n}\]
A result similar to Proposition \ref{cof:Plem} is constructed from the determinant of $D_n(k)$, which is the objective of Proposition \ref{cof:LEMG}.
\begin{proposition}\label{cof:LEMG} Let $G_n(k)$ be the matrix defined in \eqref{gen:G} and let $D_{n}(k)$ be the $n\times n$ matrix of cofactors for $G_n(k)$. Then $\la D_{n}(k)\ra = G_{k,n+1}^{n-1}$.
\end{proposition}
\begin{proof} Change $P_n(k)$ for $G_n(k)$ in the proof of Proposition \ref{cof:Plem}.
\end{proof}
Consider the first few determinants of $D_{n}(k)$ for $n = 1,2,3,4,$ given below.\begin{align*}
\la D_1(k)\ra &= P_{k,1} = G_{k,2}^0\\
\la D_2(k)\ra &= \la\begin{matrix}
P_{k,2} & P_{k,1}\\
-akP_{k,1} & G_{k,2}P_{k,1}\\
\end{matrix}\ra = G_{k,3}\\
&\implies G_{k,2}P_{k,2}+kG_{k,1}P_{k,1}= G_{k,3}\\
\la D_3(k)\ra &= \la\begin{matrix}
P_{k,3} & P_{k,2} & P_{k,1}\\
-akP_{k,2} & G_{k,2}P_{k,2} & G_{k,2}P_{k,1}\\
ak^2P_{k,1} & -kG_{k,2}P_{k,1} & G_{k,3}P_{k,1}\\
\end{matrix}\ra = G_{k,4}^2\\
&\implies (kG_{k,2}P_{k,1}+G_{k,3}P_{k,2})(kG_{k,1}P_{k,2}+G_{k,2}P_{k,3}) = G_{k,4}^2\\
\la D_4(k)\ra &= \la\begin{matrix}
P_{k,4} & P_{k,3} & P_{k,2} & P_{k,1}\\
-akP_{k,3} & G_{k,2}P_{k,3} & G_{k,2}P_{k,2} & G_{k,2}P_{k,1}\\
ak^2P_{k,2} & -kG_{k,2}P_{k,2} & G_{k,3}P_{k,2} & G_{k,3}P_{k,1}\\
-ak^3P_{k,1} & k^2G_{k,2}P_{k,1} & -kG_{k,3}P_{k,1} & G_{k,4}P_{k,1}\\
\end{matrix}\ra = G_{k,5}^3\\
&\implies (kG_{k,3}P_{k,1}+G_{k,4}P_{k,2})(kG_{k,2}P_{k,2}+G_{k,3}P_{k,3})(kG_{k,1}P_{k,3}+G_{k,2}P_{k,4}) = G_{k,5}^3
\end{align*} The generalization of these facts is the subject of Theorem \ref{cof:gthm}. To prove Theorem \ref{cof:gthm}, the following lemma, which was originally proven in \cite[Proposition 1]{7}, will be employed. 
\begin{lemma}[Binet's Formula]\label{cof:binet} Let $P_{k,n}$ denote the $n^{\text{th}}$ term of the $k$-Pell sequence. Then \[P_{k,n} = \frac{(1+\sqrt{1+k})^n-(1-\sqrt{1+k})^n}{2\sqrt{1+k}}\]
\end{lemma}
\begin{proof} The characteristic equation associated to the $k$-Pell sequence is $r^2-2r-k=0$. Solving for the roots of this equation and applying the initial conditions of the $k$-Pell sequence, $P_{k,0} = 0, P_{k,1} = 1$, completes the proof. 
\end{proof}
\begin{theorem}\label{cof:gthm} Let $P_{k,n}$ and $G_{k,n}$ denote the $n^{\text{th}}$ terms of the $k$-Pell and Generalized $k$-Pell sequences, respectively. Then \[G_{k,n+1} = kG_{k,i}P_{k,n-i}+G_{k,i+1}P_{n+1-i}\] for each $i\in\ZZ^+$.
\end{theorem}
\begin{proof}
Let $r_1 = 1+\sqrt{1+k}$ and $r_2 = 1-\sqrt{1+k}$. Note that $r_1r_2 = -k$ and $r_1-r_2 = 2\sqrt{1+k}$. Recall from Theorem \ref{thm:binet} that Binet's Formula for the Generalized $k$-Pell sequence is \[G_{k,n} = \frac{ar_1^n + ar_2^n}{2}.\] By Theorem \ref{thm:binet} and Lemma \ref{cof:binet}, the following equalities hold.\begin{align*}
kG_{k,i}P_{k,n-i}+G_{k,i+1}P_{n+1-i} &= ak\lp\frac{r_1^i + r_2^i}{2}\rp\lp\frac{r_1^{n-i}-r_2^{n-i}}{r_1-r_2}\rp\\&\hspace{4em}+a\lp\frac{r_1^{i+1} + r_2^{i+1}}{2}\rp\lp\frac{r_1^{n+1-i}-r_2^{n+1-i}}{r_1-r_2}\rp\\
&= \frac{a}{4(r_1-r_2)}[k(r_1^i + r_2^i)(r_1^{n-i}-r_2^{n-i}) \\&\hspace{4em}+ (r_1^{i+1} + r_2^{i+1})(r_1^{n+1-i}-r_2^{n+1-i})]\\
&= \frac{a}{4(r_1-r_2)}[ k(r_1^n-r_1^ir_2^{n-i}+r_1^{n-i}r_2^i-r_2^n) \\&\hspace{4em}+ (r_1^{n+2}-r_1^{i+1}r_2^{n+1-i}+r_1^{n+1-i}r_2^{i+1}-r_2^{n+2})]\\
&= \frac{a}{4(r_1-r_2)}[ r_1^n(k+r_1^2)-r_1^ir_2^{n-i}(k+r_1r_2)\\&\hspace{4em}+r_1^{n-i}r_2^i(k+r_1r_2)-r_2^n(k+r_2^2)]\\
&= \frac{a}{4\sqrt{1+k}}[ r_1^n(2+2\sqrt{1+k}+2k)-r_1^ir_2^{n-i}(k-k)\\&\hspace{4em}+r_1^{n-i}r_2^i(k-k)-r_2^n(2-2\sqrt{1+k}+2k)]\\
&= \frac{a}{2\sqrt{1+k}}\lc r_1^n(1+\sqrt{1+k}+k)-r_2^n(1-\sqrt{1+k}+k)\rc\\
&= \ts{\frac{a}{2}}\lp r_1^nr_1-r_2^n(-1)r_2\rp\\
&= \frac{ar_1^{n+1}+ar_2^{n+1}}{2}\\
&= G_{k,n+1}
\end{align*}
\end{proof}

\appendix
\section{Table of Values} \label{App:Table}
Below is a table with the first few values $n = 0,1,2,3,\dots$ of $P_{n,k}$ and $G_{n,k}$ for arbitrary $a,k\in\ZZ^+$.
\begin{center}
\begin{tabular}{|c|c|c|}
\hline
 & $P_{k,n}$ & $G_{k,n}$\\
 \hline
$n=0$ & $0$ & $a$\\
\hline
$n=1$ & $1$ & $a$\\
\hline
$n=2$ & $2$ & $ka+2a$\\
\hline
$n=3$ & $k+4$ & $3ka+4a$\\
\hline
$n=4$ & $4k+8$ & $k^2a+8ka+8a$\\
\hline
$n=5$ & $k^2+12k+16$ & $5k^2a+20ka+16a$\\
\hline
$n=6$ & $6k^2+32k+32$ & $k^3a+18k^2a+48ka+32a$\\
\hline
$n=7$ & $k^3+24k^2+80k+64$ & $7k^3a+56k^2a+112ka+64a$\\
\hline
\end{tabular}
\end{center}

\section*{Acknowledgements}
The author would like to thank Dr. Roy Wensley and the rest of the Saint Mary's College of California School of Science for the opportunity to conduct this research under the Summer Research Program of 2015. The author would also like to thank Dr. Kristen Beck for all of her help and guidance in conducting this research project.

\bibliographystyle{amsxport}
\bibliography{mybib}

\begin{bibdiv}
\begin{biblist}

\bib{7}{article}{
      author={Catarino, Paula},
       title={On some identities and generating functions for {$k$}-{P}ell
  numbers},
        date={2013},
        ISSN={1312-8876},
     journal={Int. J. Math. Anal. (Ruse)},
      volume={7},
      number={37-40},
       pages={1877\ndash 1884},
      review={\MR{3084999}},
}

\bib{1}{article}{
      author={Catarino, Paula},
       title={On generating matrices of the {$k$}-{P}ell, {$k$}-{P}ell-{L}ucas
  and modified {$k$}-{P}ell sequences},
        date={2014},
     journal={Pure Mathematical Sciences},
      volume={3},
      number={2},
       pages={71\ndash 77},
}

\bib{6}{article}{
      author={Catarino, Paula},
      author={Vasco, Paulo},
       title={On some identities and generating functions for
  {$k$}-{P}ell-{L}ucas sequence},
        date={2013},
        ISSN={1312-885X},
     journal={Appl. Math. Sci. (Ruse)},
      volume={7},
      number={97-100},
       pages={4867\ndash 4873},
      review={\MR{3119767}},
}

\bib{5}{article}{
      author={Catarino, Paula},
      author={Vasco, Paulo},
       title={Some basic properties and a two-by-two matrix involving the
  {$k$}-{P}ell numbers},
        date={2013},
        ISSN={1312-8876},
     journal={Int. J. Math. Anal. (Ruse)},
      volume={7},
      number={45-48},
       pages={2209\ndash 2215},
      review={\MR{3125807}},
}

\bib{2}{article}{
      author={Falcon, Sergio},
       title={On the generating matrices of the {$k$}-{F}ibonacci numbers},
        date={2013},
        ISSN={0716-0917},
     journal={Proyecciones},
      volume={32},
      number={4},
       pages={347\ndash 357},
         url={http://dx.doi.org/10.4067/S0716-09172013000400004},
      review={\MR{3145040}},
}

\bib{3}{article}{
      author={Lakoba, Taras~I.},
       title={The eigenvalues of tridiagonal matrices},
        date={Autumn 2009},
     journal={course notes for TMA 4205 at Norwegian University of Science and
  Technology, Department of Mathematical Sciences},
        note={Available at
  \url{http://www.cems.uvm.edu/~tlakoba/math337/proof_eigensystem_tridiagonal.pdf}},
}

\bib{4}{article}{
      author={Usmani, R.~A.},
       title={Inversion of {J}acobi's tridiagonal matrix},
        date={1994},
        ISSN={0898-1221},
     journal={Comput. Math. Appl.},
      volume={27},
      number={8},
       pages={59\ndash 66},
         url={http://dx.doi.org/10.1016/0898-1221(94)90066-3},
      review={\MR{1270775 (94k:65050)}},
}

\end{biblist}
\end{bibdiv}

\end{document}